\documentclass[a4paper,reqno,11pt]{amsart}

\usepackage{amsmath, amsfonts, amssymb, amsthm, amscd}
\usepackage{graphicx}
\usepackage{psfrag}
\usepackage{perpage}
\usepackage{url}
\usepackage{color}
\usepackage{mathrsfs}
\usepackage{tikz}
\usepackage{mathabx}
\usetikzlibrary{arrows,decorations.pathmorphing,backgrounds,positioning,fit,petri} 
\usepackage{tikz,tikz-cd}\usepackage{caption,subcaption}

\usepackage{dsfont} 

\usepackage[utf8]{inputenc}
\usepackage[T1]{fontenc}
\usepackage{microtype}

\usepackage[a4paper,scale={0.72,0.74},marginratio={1:1},footskip=7mm,headsep=10mm]{geometry}

\usepackage{hyperref}

\makeatletter
\def\@secnumfont{\bfseries\scshape}

\def\section{\@startsection{section}{1}%
  \z@{.7\linespacing\@plus\linespacing}{.5\linespacing}%
  {\normalfont\large\bfseries\scshape\centering}}

\def\subsection{\@startsection{subsection}{2}%
  \z@{.5\linespacing\@plus.7\linespacing}{-.5em}%
  {\normalfont\bfseries\scshape}}

\def\subsubsection{\@startsection{subsubsection}{3}%
  \z@{.5\linespacing\@plus.7\linespacing}{-.5em}%
  {\normalfont\scshape}}

\def\specialsection{\@startsection{section}{1}%
  \z@{\linespacing\@plus\linespacing}{.5\linespacing}%
  {\normalfont\centering\large\bfseries\scshape}}
\makeatother

\makeatletter

\renewenvironment{proof}[1][\proofname]{\par
\pushQED{\qed}%
\normalfont \topsep4\p@\@plus4\p@\relax
\trivlist
\item[\hskip\labelsep
\bfseries
#1\@addpunct{.}]\ignorespaces
}{%
\popQED\endtrivlist\@endpefalse
}
\makeatother

\makeatletter
\newcommand \Dotfill {\leavevmode \leaders \hb@xt@ 6pt{\hss .\hss }\hfill \kern \z@}
\makeatother

\makeatletter
\def\@tocline#1#2#3#4#5#6#7{\relax
  \ifnum #1>\c@tocdepth 
  \else
    \par \addpenalty\@secpenalty\addvspace{#2}%
    \begingroup \hyphenpenalty\@M
    \@ifempty{#4}{%
      \@tempdima\csname r@tocindent\number#1\endcsname\relax
    }{%
      \@tempdima#4\relax
    }%
    \parindent\z@ \leftskip#3\relax \advance\leftskip\@tempdima\relax
    \rightskip\@pnumwidth plus4em \parfillskip-\@pnumwidth
    #5\leavevmode\hskip-\@tempdima
      \ifcase #1
       \or\or \hskip 1.65em \or \hskip 3.3em \else \hskip 4.95em \fi%
      #6\nobreak\relax
    \Dotfill
    \hbox to\@pnumwidth{\@tocpagenum{#7}}\par
    \nobreak
    \endgroup
  \fi}
\makeatother

\makeatletter
\def\l@section{\@tocline{1}{0pt}{1pc}{}{\scshape}}
\renewcommand{\tocsection}[3]{%
\indentlabel{\@ifnotempty{#2}{\ignorespaces#1 #2.\hskip 0.7em}}#3}
\def\l@subsection{\@tocline{2}{0pt}{1pc}{5pc}{}}

\def\l@subsubsection{\@tocline{3}{0pt}{1pc}{7pc}{}}

\makeatother

\numberwithin{equation}{section}

\newtheoremstyle{mytheorem}{.7\linespacing\@plus.3\linespacing}{.7\linespacing\@plus.3\linespacing}%
     {\itshape}
     {}
     {\bfseries}
     {. }
     {0.3ex}
     {\thmname{{\bfseries #1}}\thmnumber{ {\bfseries #2}}\thmnote{ (#3)}}  

\theoremstyle{mytheorem}

\newtheorem{theorem}{Theorem}[section]
\newtheorem{lemma}[theorem]{Lemma}

\newtheorem{corollary}[theorem]{Corollary}
\newtheorem{remark}[theorem]{Remark}
\newtheorem{definition}[theorem]{Definition}

\renewcommand{\tilde}{\widetilde}          
\DeclareMathSymbol{\leqslant}{\mathalpha}{AMSa}{"36} 
\DeclareMathSymbol{\geqslant}{\mathalpha}{AMSa}{"3E} 
\DeclareMathSymbol{\eset}{\mathalpha}{AMSb}{"3F}     

\newcommand\blue{{}}

\newcommand{\be}{\begin{equation}}
\newcommand{\ee}{\end{equation}}

\newcommand{\R}{\mathbb{R}}

\newcommand{\Z}{\mathbb{Z}}
\newcommand{\N}{\mathbb{N}}

\newcommand{\PEfont}{\mathrm}

\newcommand{\p}{\ensuremath{\PEfont P}}

\renewcommand{\P}{\p}

\renewcommand{\epsilon}{\varepsilon}
\renewcommand{\rho}{\varrho}

\newenvironment{myenumerate}{%
\renewcommand{\theenumi}{\arabic{enumi}}%
\renewcommand{\labelenumi}{{\rm(\theenumi)}}%
\begin{list}{\labelenumi}
	{%
	\setlength{\itemsep}{0.4em}%
	\setlength{\topsep}{0.5em}%
	\setlength\leftmargin{2.45em}%
	\setlength\labelwidth{2.05em}%
	\setlength{\labelsep}{0.4em}%
	\usecounter{enumi}%
	}%
	}%
{\end{list}
}

{\end{list}
}

{\end{list}
}

{\end{myenumerate}}

\newenvironment{myitemize}{%
\begin{list}{$\bullet$}%
 	{%
	\setlength{\itemsep}{0.4em}%
	\setlength{\topsep}{0.5em}%
	\setlength\leftmargin{2.65em}%
	\setlength\labelwidth{2.65em}%
	\setlength{\labelsep}{0.4em}%
	}%
	}%
{\end{list}}

\renewenvironment{itemize}{
\begin{myitemize}}%
{\end{myitemize}}

\MakePerPage[2]{footnote} 

\date{\today}

\title[A New Correlation Inequality for Ising Models]{A New Correlation Inequality for Ising models with external fields}

\author[Jian Ding]{Jian Ding}
\address{School of Mathematical Sciences\\ 
Peking University}
\email{dingjian@math.pku.edu.cn}

\author[Jian Song]{Jian Song}
\address{Research Center for     Mathematics and Interdisciplinary Sciences, Shandong University; 
 School of Mathematics, Shandong University}
\email{txjsong@sdu.edu.cn}

\author[R. Sun]{Rongfeng Sun}
\address{Department of Mathematics\\
National University of Singapore}
\email{matsr@nus.edu.sg}

\begin{document}

\begin{abstract}
We study ferromagnetic Ising models on finite graphs with an inhomogeneous external field, where a subset of vertices is designated as the boundary. We show that the influence of boundary conditions on any given spin is maximised when the external field is identically $0$. One corollary is that spin-spin correlations are maximised when the external field vanishes and the boundary condition is free, which proves a conjecture of Shlosman. In particular, the random field Ising model on $\Z^d$, $d\geq 3$, exhibits exponential decay of correlations in the entire high temperature regime of the pure Ising model. Another corollary is that the pure Ising model in $d\geq 3$ satisfies the conjectured strong spatial mixing property in the entire high temperature regime.
\end{abstract}


\maketitle


\section{Introduction}

The Ising model was introduced by Lenz \cite{L20} as a model for magnets, which became the most important statistical physics model in the study of phase transitions. Correlation inequalities are powerful tools in the study of the Ising model, although they are hard to come by. The most important inequalities include the FKG inequality~\cite{FKG71} which establishes positive correlation between increasing functions of the spin configuration, the GKS inequalities~\cite{G67, KS68} which concern the expected value of the product of spins; the GHS inequality~\cite{GHS70} which gives the concavity of the expected spin value in a positive external field. More background can be found in books such as \cite{FV17, B06}. In this paper, we prove a new correlation inequality for the Ising model, which allows us to compare an Ising model with arbitrary external fields to one without an external field, and compare arbitrary boundary conditions with the free boundary condition. Heuristically, the presence of external fields, regardless of their sign, weakens the correlation between spins as well as the influence of boundary conditions on interior spins. We will turn this heuristic into a rigorous statement.

We believe our correlation inequality provides a powerful new tool for the study of the Ising model. As corollaries, we show that the random field Ising model exhibits exponential decay of correlations whenever the temperature is above the critical temperature of the pure Ising model. We also prove the long standing conjecture that the Ising model on $\Z^d$, $d\geq 3$, satisfies the strong spatial mixing property whenever the temperature is above the critical value. We will discuss more details later. First we introduce the basic setup.

Let $G=(V, E)$ be a finite graph with vertex set $V$ and edge set $E$. If $u,v\in V$ are connected by an edge in $E$, we denote it by $u\sim v$. Given coupling constants $J: E\to [0,\infty)$, inverse temperature $\beta \in (0,\infty)$, and external field $g: V \to [-\infty, \infty]$, we define the Ising model on spin configurations $\sigma \in \{\pm 1\}^V$ via the Gibbs measure
\begin{equation}
\mu_g(\sigma) = \frac{1}{Z_g} \exp\Big\{\beta\sum_{u\sim v} J_{uv} \sigma_u \sigma_v + \sum_{u\in V} g_u \sigma_u\Big\},
\end{equation}
where $Z_g$ is the normalizing constant known as the partition function. Since $\beta$ can be absorbed by the coupling constants $J$, we will assume $\beta=1$ unless specified otherwise. For more background on the Ising model, see e.g.~\cite{B06, FV17}. We will denote expectation with respect to $\mu_g$ by $\langle \cdot \rangle_g$. Note that  when $g_u=\pm \infty$ for some $u\in V$, we must have $\sigma_u=\pm1$, which effectively imposes a boundary condition at $u$ (this boundary condition has the same effect as removing $u$ from the graph and adding an extra field $\pm J_{uv}$ to each neighbour $v\sim u$ if $g_u=\pm \infty$). To impose $+$ or $-$ boundary conditions on a subset of vertices $V'\subset V$, we can start with $g:V\to\R$, set $h\equiv \infty$ on $V'$ and $h\equiv 0$ on $V\backslash V'$, and then replace $g$ by $g^+_{V'}:=g+h$ for $+$ boundary condition on $V'$, or $g^-_{V'}:=g-h$ for $-$ boundary condition on $V'$. We can then define the influence of the boundary spins $(\sigma_u)_{u\in V'}$ on a given vertex $o\in V\backslash V'$ by
$$
\langle \sigma_o\rangle_{g^+_{V'}} - \langle \sigma_o\rangle_{g^-_{V'}}.
$$
Our main result is that
\begin{equation}\label{eq:inf0}
\langle \sigma_o\rangle_{g^+_{V'}} - \langle \sigma_o\rangle_{g^-_{V'}} \leq \langle \sigma_o\rangle_{0^+_{V'}} - \langle \sigma_o\rangle_{0^-_{V'}}.
\end{equation}
Namely, the influence of the boundary spins $(\sigma_v)_{v\in V'}$ on $\sigma_o$ is maximized when the external fields on the interior spins $(\sigma_v)_{v\in V\backslash V'}$ are identically $0$. 

We will in fact prove the following more general result where $h$ can take arbitrary values in $[0,\infty]$, not just in $\{0, \infty\}$. 
\begin{theorem}\label{T:inf} Let $g: V\to [-\infty, \infty]$ and $h: V\to [0,\infty]$ be such that $\min\{|g_v|, h_v\}<\infty$ for all $v\in V$. Then for any $o\in V$,
\begin{equation}\label{eq:inf}
\langle \sigma_o\rangle_{g+h} - \langle \sigma_o\rangle_{g-h} \leq \langle \sigma_o\rangle_{h} - \langle \sigma_o\rangle_{-h}.
\end{equation}
\end{theorem}
\noindent
When $V$ is a tree and $h_v=\infty$ at some $v\in V$ and $0$ everywhere else,  this result has been proved by N.~Berger, C.~Kenyon, E.~Mossel, and Y.~Peres in \cite[Lemma 4.1]{BKMY05}.

\begin{remark}\label{rem1}
\blue{It is tempting to conjecture that $\langle \sigma_o\rangle_{\lambda g+h} - \langle \sigma_o\rangle_{\lambda g-h}$ is monotonically decreasing in $\lambda \geq 0$. However, this is not true in general. We will give a counterexample in Appendix \ref{S:example}. }
\end{remark}

As a first corollary of Theorem \ref{T:inf}, we show that for the Ising model on finite graphs, spin-spin correlations are maximised when the external field vanishes.
As communicated to us by R.~van den Berg, S.~Shlosman first conjectured this inequality and proved it for certain small graphs. 
\begin{corollary}\label{C:corr}
Let $g: V\to [-\infty, \infty]$. Then for any $u, v\in V$,
\begin{equation}\label{eq:corr}
\langle \sigma_u \sigma_v\rangle_g - \langle \sigma_u\rangle_g \langle \sigma_v\rangle_g \leq \langle \sigma_u\sigma_v\rangle_0. 
\end{equation}
\end{corollary}
\noindent
Note that boundary conditions can also be regarded as external fields imposed on the interior spins. Therefore Corollary \ref{C:corr} implies that spin correlations is maximised under free boundary conditions (if $V=U\cup \partial U$ with $U\cap \partial U=\emptyset$ and boundary conditions are imposed on $\partial U$, then free boundary condition for $U$ means removing $\partial U$ from the graph). 

\begin{remark}\label{rmk1.4}
\blue{It is also tempting to conjecture that the l.h.s.\ of \eqref{eq:corr} is decreasing in $\lambda\geq 0$ if $g=\lambda \tilde g$, which would imply \eqref{eq:corr}. However, this is not true in general either. We will give a counterexample in Appendix \ref{S:example}. If $\tilde g\geq 0$, then there is indeed monotonicity in $\lambda\geq 0$, which follows from the GHS inequality \cite{GHS70} since the l.h.s.\ of \eqref{eq:corr} equals $\frac{\partial^2 \log Z_g}{\partial g_u \partial g_v}$. 
}
\end{remark}

\begin{remark}\label{rmk1.5}
\blue{Theorem \ref{T:inf} and Corollary \ref{C:corr} in fact hold for general Ising models whose single spin measure, i.e., the a priori distribution of each $\sigma_u$, belongs to the so-called Griffiths-Simon class $($see \cite{G69, SG73} or \cite[Section 2]{ADC21}$)$, where the single spin measure can be approximated by that of a weighted average of a block of $\pm1$-valued Ising spins with ferromagnetic interactions. This includes in particular the $\phi^4$ model.\footnote{\blue{Recently, Corollary \ref{C:corr} has been used by Bauerschmidt and Dagallier to prove log-Sobolev inequalities for the $\phi^4_2$ and $\phi^4_3$ models \cite{BD22b} and the near-critical Ising models \cite{BD22a}.}} Such an extension holds because \eqref{eq:inf} and \eqref{eq:corr} are preserved if each $\sigma_x$ therein is replaced by a positive linear combination of Ising spins in the system.}
\end{remark}

The second corollary of Theorem \ref{T:inf} concerns the random field Ising model (RFIM), where the external field (as a function on the vertex set $V$) is given by a family of i.i.d.\ random variables $(\omega_v)_{v\in V}$ (see \cite[Chapter 7]{B06} for an overview). The most important case  is when $V=[-N, N]^d \cap \Z^d$, with boundary conditions imposed on the surface of the cube. Since phase transitions only arise in the infinite volume limit, a fundamental question is whether the influence of boundary conditions persists or vanishes as $N\to\infty$, which determines whether or not the RFIM undergoes a first order phase transition. Classic results include the work of Aizenman and Wehr~\cite{AW90}, who showed that in $d=2$, as long as there is randomness in the external field, the boundary influence always decays to $0$ as $N\to\infty$. Recently, this decay was shown to be exponential by Ding and Xia~\cite{DX21} and also Aizenman, Harel and Peled~ \cite{AHP20} for Gaussian random fields. In contrast, in $d\geq 3$, Bricmont and Kupiainen \cite{BK88} showed that when the distribution of the random field $\omega_v$ has a sufficiently fast decaying Gaussian tail, then at sufficiently low temperatures, the boundary influence does not vanish as $N\to\infty$. But if the temperature is sufficiently high, or if the random field is sufficiently large, then again the boundary influence decays exponentially fast (see \cite{IF84, B85, DKP95, CJN18}). Among these results, only \cite{CJN18} gave quantitative bounds on the parameter regime that has exponential decay. But it still falls short of covering the entire high temperature regime of the pure Ising model. This now follows as an immediate corollary of Theorem \ref{T:inf}.

\begin{corollary}\label{C:RFIM}
Let $\Lambda_N:= [-N, N]^d \cap \Z^d$ with $d\geq 2$. Let $\omega: \Z^d\to [-\infty, \infty]$ be an arbitrary external field. Let $\langle \cdot \rangle_{\Lambda_N, \beta}^{\omega, \pm}$ $($with $\omega$ omitted if $\omega\equiv 0)$ denote expectation w.r.t.\ the Ising model on $\Lambda_N$ with coupling constants $J\equiv 1$, inverse temperature $\beta$, external field $\omega$, and $+/-$ boundary condition on $\partial \Lambda_N:=\{x\in \Z^d\backslash \Lambda_N: |x-y|_1=1$ for some $y\in \Lambda_N\}$. Then for all $N\in\N$,
\begin{equation}\label{eq:RFIM}
\langle \sigma_0 \rangle_{\Lambda_N, \beta}^{\omega, +} - \langle \sigma_0 \rangle_{\Lambda_N, \beta}^{\omega, -} \leq \langle \sigma_0 \rangle_{\Lambda_N, \beta}^{+} - \langle \sigma_0 \rangle_{\Lambda_N, \beta}^{-} \leq C_1(\beta) e^{-C_2(\beta)N},
\end{equation}
where $C_2(\beta)=0$ for $\beta\geq \beta_c$ and $C_2(\beta)>0$ for $\beta <\beta_c$, with $\beta_c$ being the critical inverse temperature of the Ising model on $\Z^d$.
\end{corollary}
\noindent
The first inequality in \eqref{eq:RFIM} is an immediate consequence of Theorem \ref{T:inf}, while the second inequality is a classic result by Aizenman, Barsky, and Fern\'andez~\cite{ABF87} (see also \cite{DT16, DRT19}).
\medskip

The third corollary of Theorem \ref{T:inf} is the strong spatial mixing property of the Ising model on $\Z^d$, $d\geq 2$, in the entire high temperature regime. Given a finite domain $V\subset \Z^d$ with boundary condition $\tau: \partial V \to \{\pm 1\}$ on the external boundary $\partial V$, external field $h:V\to\R$, coupling constants $J\equiv 1$ and inverse temperature $\beta$, the associated Ising measure $\mu^{\tau,h}_{V, \beta}$ is said to satisfy the weak spatial mixing property if for any $\Delta \subset V$, the influence of the boundary condition on $(\sigma_v)_{v\in \Delta}$ decays exponentially in $d(\Delta, \partial V):=\min_{x\in \Delta, y\in \partial V} \vert x-y\vert_1$ (see \cite[Section 2.3]{M99} for a precise definition). The strong spatial mixing property requires instead exponential decay in $d(\Delta, S)$, where $S$ is the subset of $\partial V$ where the spins in $\tau$ are flipped. More precisely,

\begin{definition}
The Ising measures $\mu^{\cdot, h}_{V, \beta}$ is said to satisfy the strong spatial mixing property in $V$ with constants $C$ and $m$ $($denoted by $SM(V, C, m))$, if  for every $\Delta \subset V$ and $y\in \partial V$,
\begin{equation}\label{eq:mixing}
\sup_{\tau: \partial V \to \{\pm1\}} \Vert \mu^{\tau,h}_{V, \beta}|_\Delta - \mu^{\tau^y, h}_{V, \beta}|_\Delta\Vert_{TV} \leq C e^{-m d(\Delta, y)},
\end{equation}
where $\mu^{\tau, h}_{V, \beta}|_\Delta$ denotes the marginal law of $(\sigma_v)_{v\in \Delta}$ under $\mu^\tau_{V, \beta}$, $\tau^y$ is obtained from $\tau$ by flipping the spin $\tau_y$, and $\Vert \cdot\Vert_{TV}$ denotes total variation distance between measures.
\end{definition}
As a corollary of Theorem \ref{T:inf}, we have the following result.
\begin{corollary}\label{C:mixing}
Let $d\geq 2$. Then for any $\beta \in [0, \beta_c)$, there exist $C, m\in (0,\infty)$ such that the Ising measures $\mu^{\cdot, h}_{V, \beta}$ satisfy $SM(V, C, m)$ for all finite $V\subset \Z^d$ and $h:V\to \R$.
\end{corollary}
Weak and strong spatial mixing play a key role in the study of exponential ergodicity of Glauber dynamics for the Ising model (see \cite[Section 2.3]{M99}). \blue{In $d=2$, when there is a constant external field $h\neq 0$ or when $\beta <\beta_c$, strong spatial mixing has been established for all square domains in \cite{MOS94} and for a more general class of domains in 
\cite{SS95}.} Our result for all finite domains when $\beta<\beta_c$ appears to be new. In $d\geq 3$, strong spatial mixing has only been proved for all cubic domains if $\beta$ is sufficiently small or $\beta h$ is sufficiently large, but it is not expected to hold for $\beta$ large and $h$ sufficiently small~\cite{MO94}. Our result proves the conjecture that strong spatial mixing holds in the entire high temperature regime for all finite domains.

In recent years, strong spatial mixing has also played a crucial role in studying cutoff of Glauber dynamics for the Ising model, namely, a sharp $L^1$ transition to equilibrium. Lubetzky and Sly showed in~\cite{LS13, LS14} that the Glauber dynamics for the Ising model exhibits cutoff as long as the Ising measures satisfy strong spatial mixing for all cubic domains, which applies in particular in $d=2$ in the entire high temperature regime. In $d\geq 3$, strong spatial mixing was only known at the time for $\beta$ sufficiently small. So in \cite{LS16}, they developed an alternative framework, called information percolation, to prove cutoff in the entire high temperature regime in $d\geq 3$. We can now deduce this cutoff directly by combining Corollary \ref{C:mixing} with results from \cite{LS13, LS14}.

Lastly, we remark that Theorem \ref{T:inf} provides exactly the type of estimates needed to apply Theorem 1.3 in the recent work \cite{ALG20} to get quantitative bounds on the spectral gap for Glauber dynamics in the Ising model.

\section{Proof of Theorem \ref{T:inf}}

\begin{proof}[Proof of Theorem \ref{T:inf}]
First note that the l.h.s.\ of \eqref{eq:inf}, $\langle \sigma_o\rangle_{g+h} - \langle \sigma_o\rangle_{g-h}$, is continuous in $(g_v)_{v\in V} \in [-\infty, \infty]^V$.
Therefore it suffices to prove \eqref{eq:inf} for all $g: V\to \R$, which we assume from now on. We will prove this by a two-dimensional induction on $(n, m)=(|V|, |V_+|)$, the cardinalities of $V$ and $V_+:=\{v\in V: h_v>0\}$.

\blue{For $n\in\N$ and $0\leq m\leq n$, let ${\rm P}(n,m)$ denote the following claim: 
$$
{\rm P}(n,m): \quad  \mbox{\eqref{eq:inf} holds for all $G$, $g:V\to\R$, and $h:V\to [0,\infty]$ with $(|V|, |V_+|)=(n, m)$.}
$$}
To initiate the induction, note that when $|V_+|=0$, that is, $h\equiv 0$, \eqref{eq:inf} always holds because both sides equal $0$. When $|V|=|V_+|=1$, let $o$ denote the only vertex in $V$ and omit it from the subscript of $g$ and $h$. By the definition of $\mu_{g+h}$, we have
\begin{equation}\label{eq:1spin}
\langle \sigma_o\rangle_{g+h} = \frac{e^{g+h} - e^{-(g+h)}}{e^{g+h} + e^{-(g+h)}} = \tanh (g+h).
\end{equation}
Then \eqref{eq:inf} is equivalent to
\begin{equation}\label{eq:gpmh}
f(g):= \tanh (g+h)-\tanh (g-h) \leq \tanh h -\tanh (-h),
\end{equation}
where the inequality is in fact strict when $g\neq 0$. This holds because $f(g)$ is an even function, and for $g\geq 0$,
$$
f'(g) = \frac{1}{\cosh^2(g+h)} - \frac{1}{\cosh^2(g-h)} = \frac{\cosh(2g-2h) -\cosh(2g+2h)}{2\cosh^2(g+h) \cosh^2(g-h)}\leq 0,
$$
\blue{where the inequality holds since $\cosh(2g-2h) \leq \max\{\cosh 2g, \cos 2h\} \leq \cosh (2g+2h)$.} 

\blue{
We have thus shown that ${\rm P}(n,m)$ holds for all pairs in $\{(n, m): n\in \N, m=0\} \cup \{(1,1)\}$, which proves the base cases for induction.
We then note that to prove Theorem \ref{T:inf}, it only remains to show the following induction step:
\begin{itemize}
\item[\bf ($\dagger$)] Given $(n,m)$ with $n\geq 2$ and $1\leq m\leq n$, if $\{{\rm P}(n,i): 0\leq i\leq m-1\}$ and $\{{\rm P}(n-1,i): 0\leq i\leq n-1\}$ all hold, then
${\rm P}(n,m)$ also holds.
\end{itemize}
Applying this induction step repeatedly will establish ${\rm P}(n,m)$ for all $n\in\N$ and $0\leq m\leq n$. 
}

Let $G$, $g:V\to\R$, and $h:V\to [0,\infty]$ be arbitrary with $|V|=n\geq 2$ and $1\leq |V_+|=m\leq n$. Suppose that $\{{\rm P}(n,i): 0\leq i\leq m-1\}$ and $\{{\rm P}(n-1,i): 0\leq i\leq n-1\}$
all hold. For arbitrary $o\in V$, we will show that \eqref{eq:inf} holds, which will conclude the induction step {\bf ($\dagger$)}.

First consider the case $V_+=\{o\}$, so that $h_v=0$ for all $v\neq o$. We then have
\begin{equation}\label{eq:efield}
\langle \sigma_o\rangle_{g+h} =\tanh (\lambda(g)+g_o+h_o) \quad \mbox{and} \quad \langle \sigma_o\rangle_{g-h} =\tanh (\lambda(g)+g_o-h_o),
\end{equation}
where $\lambda(g)$ is the effective field induced by other spins on $\sigma_o$; more precisely, if we set $\tilde g:=g$ except that $\tilde g_o:=0$, then the marginal distribution of $\sigma_o$ under $\mu_{\tilde g}$ is the same as in a single spin system with external field $\lambda(g)$ (compare with \eqref{eq:1spin}), with \begin{equation}\label{eq:efi}
e^{2\lambda(g)} = \frac{\mu_{\tilde g}(\sigma_o=+1)}{\mu_{\tilde g}(\sigma_o=-1)}=
\frac{ \sum_{\sigma: \sigma_o=+1}\exp\Big\{\sum_{u\sim v} J_{uv} \sigma_u \sigma_v + \sum_{u\in V\backslash\{o\}} g_u \sigma_u\Big\}}{ \sum_{\sigma: \sigma_o=-1} \exp\Big\{\sum_{u\sim v} J_{uv} \sigma_u \sigma_v + \sum_{u\in V\backslash\{o\}} g_u \sigma_u\Big\}}.
\end{equation}
This then implies $ \frac{\mu_{g\pm h}(\sigma_o=+1)}{\mu_{g\pm h}(\sigma_o=-1)}=e^{2(\lambda(g) +g_o\pm h_o)}$ and \eqref{eq:efield}. Note that when $g\equiv 0$, $\lambda(g)=0$ by symmetry. Since we have proved \eqref{eq:inf} for the case $|V|=|V_+|=1$, we have
\begin{equation}
\begin{aligned}
\langle \sigma_o\rangle_{g+h} - \langle\sigma_o\rangle_{g-h} & = \tanh (\lambda(g)+g_o+h_o)-\tanh (\lambda(g)+g_o-h_o) \\
& \leq \tanh (h_o)-\tanh (-h_o) = \langle \sigma_o\rangle_{h} - \langle\sigma_o\rangle_{-h}.
\end{aligned}
\end{equation}
This proves \eqref{eq:inf} for the case $V_+=\{o\}$.

Now consider the case $V_+\neq \{o\}$ and $|V_+|\geq 1$. There must exist $v\in V\backslash\{o\}$ with $h_v>0$. We can perform the decomposition
\begin{equation}
\begin{aligned}\label{eq:sig1}
\langle \sigma_o\rangle_{\pm h} & = \mu_{\pm h}(\sigma_v=1)\langle \sigma_o |\sigma_v=1\rangle_{\pm h}  + \mu_{\pm h}(\sigma_v=-1) \langle \sigma_o |\sigma_v=-1\rangle_{\pm h}, \\
\langle \sigma_o\rangle_{g\pm h} & = \mu_{g\pm h}(\sigma_v=1)\langle \sigma_o |\sigma_v=1\rangle_{g\pm h}  + \mu_{g\pm h}(\sigma_v=-1) \langle \sigma_o |\sigma_v=-1\rangle_{g\pm h}, \\
\end{aligned}
\end{equation}
where $\langle \cdot |\cdot\rangle_{\cdot}$ denotes conditional expectation.

Before giving the actual proof, we first give some heuristics. To be able to apply the induction hypothesis, the key idea is to replace the marginal distribution of $\sigma_v$ under $\mu_h$ by that of a mixture between $\mu_{h}^{(0)}$ and $\mu_{h}^{(\infty)}$, where $\mu_{h}^{(a)}$ denotes the measure with external field $h$, but $h_v$ is reset to the value $a$. More precisely, there exists a unique $\alpha \in [0,1]$ such that
\begin{equation}\label{eq:sig2}
\begin{aligned}
\langle \sigma_v\rangle_h & = \alpha \langle \sigma_v\rangle_h^{(0)} + (1-\alpha) \langle \sigma_v\rangle_h^{(\infty)} = \alpha \langle \sigma_v\rangle_h^{(0)} + (1-\alpha), \\
\langle \sigma_v\rangle_{-h} & = \alpha \langle \sigma_v\rangle_{-h}^{(0)} + (1-\alpha) \langle \sigma_v\rangle_{-h}^{(-\infty)} = \alpha \langle \sigma_v\rangle_{-h}^{(0)} - (1-\alpha),
\end{aligned}
\end{equation}
where $\langle \cdot\rangle_{\pm h}^{(a)}$ denotes expectation with respect to $\mu_{\pm h}^{(a)}$, and we used that $\sigma_v=\pm 1$ when the field at $v$ is $\pm \infty$. Since $\langle \sigma_v\rangle_h = 2\mu_h(\sigma_v=1)-1$, the decomposition in the first identity in \eqref{eq:sig2} remains equivalent if we replace $\langle \sigma_v\rangle_h^{(\cdot)}$ by $\mu_h^{(\cdot)}(\sigma_v=1)$. In particular, the marginal distribution of $\sigma_v$ under $\mu_h$ is a mixture of the law of $\sigma_v$ under $\mu_h^{(0)}$ and $\mu_h^{(\infty)}$, with coefficients $\alpha$ and $1-\alpha$ respectively. Substituting this decomposition into \eqref{eq:sig1} and using the fact that $\langle \sigma_o |\sigma_v=\pm 1\rangle_h$ does not depend on $h_v$, we obtain
\begin{equation}\label{eq:sig3}
\langle \sigma_o\rangle_{\pm h} = \alpha \langle \sigma_o\rangle^{(0)}_{\pm h} + (1-\alpha) \langle \sigma_o\rangle^{(\pm \infty)}_{\pm h}.
\end{equation}

If we can find some $H$ such that the law of $\sigma_v$ under $\mu_{g\pm h}$ can be similarly decomposed as the mixture of the law of $\sigma_v$ under $\mu_{g\pm h}^{(H)}$ (the same $H$ in both $\mu_{g\pm h}^{(H)}$) and $\mu_{g\pm h}^{(\pm \infty)}$ with the same mixture coefficient $\alpha$, then substituting this decomposition into \eqref{eq:sig1} gives
\begin{equation}\label{eq:sig4}
\langle \sigma_o\rangle_{g\pm h}  = \alpha \langle \sigma_o\rangle^{(H)}_{g\pm h} + (1-\alpha) \langle \sigma_o\rangle_{g\pm h}^{(\pm \infty)}.
\end{equation}
The desired bound \eqref{eq:inf} would then follow from
\begin{equation}\label{eq:2ineq}
\begin{aligned}
\langle \sigma_o\rangle^{(H)}_{g+h} - \langle \sigma_o\rangle^{(H)}_{g-h} & \leq \langle \sigma_o\rangle^{(0)}_{h} - \langle \sigma_o\rangle^{(0)}_{-h}, \\
\langle \sigma_o\rangle^{(\infty)}_{g+h} - \langle \sigma_o\rangle^{(-\infty)}_{g-h} & \leq \langle \sigma_o\rangle^{(\infty)}_{h} - \langle \sigma_o\rangle^{(-\infty)}_{-h},
\end{aligned}
\end{equation}
both of which follow from the assumption in {\bf ($\dagger$)} \blue{that $\{{\rm P}(n,i): 0\leq i\leq m-1\}$ and $\{{\rm P}(n-1,i): 0\leq i\leq n-1\}$ all hold,} because in the first case, $|V_+|$ has been reduced by $1$ since the field at $v$ has been set to the common value $H$. In the second case, $\sigma_v$ has the same effect as adding an additional field of $\pm J_{uv}$ to all the neighbours $u\sim v$ if the field at $v$ is $\pm \infty$, which changes $h_u$ to $h_u+J_{uv}$. With this change in the field $h$, we can remove $v$ from the graph and \blue{reduce $|V|$ by $1$}.

Unfortunately, we cannot expect \eqref{eq:sig4} to hold with the same choice of $H$ for both $\langle \sigma_o\rangle_{g\pm h}$. \blue{However, we only need to bound the change of the l.h.s.\ of \eqref{eq:sig4} by the change of its r.h.s.\ as we change the field from $g-h$ to $g+h$. This can be accomplished with the help of the following lemma.}

\begin{lemma}\label{L:ineq}
Assume that $G$, $g:V\to\R$, $h:V\to [0,\infty]$ satisfy $|V|=n\geq 2$, $1\leq |V_+|=m\leq n$, and $h_v>0$ for some $v\in V$.  Suppose that the assumptions in {\bf $(\dagger)$} hold for $(n,m)$. Let $\alpha$ be defined as in \eqref{eq:sig2}. Then for all $c_+, c_-\geq 0$, we can find $H$ such that
\begin{equation}\label{eq:ineq}
c_+ \langle \sigma_v\rangle_{g+ h}- c_- \langle \sigma_v\rangle_{g-h}
\leq  c_+ \Big( \alpha \langle \sigma_v\rangle_{g+ h}^{(H)} + 1-\alpha\Big) - c_- \Big(\alpha \langle \sigma_v\rangle_{g-h}^{(H)} -(1-\alpha)\Big) .
\end{equation}
\end{lemma}
We now show how Lemma \ref{L:ineq} can be used to prove \eqref{eq:inf} and conclude the induction step {\bf ($\dagger$)}. Using \eqref{eq:sig1} and $\langle \sigma_v\rangle_{g\pm h} = 2 \mu_{g\pm h}(\sigma_v=1)-1 = 1-2 \mu_{g\pm h}(\sigma_v=-1)$, we can write
\begin{align}
& \langle \sigma_o\rangle_{g\pm h} \label{eq:sig5}\\
=\ &  \frac{1}{2}\Big(\langle \sigma_o |\sigma_v=1\rangle_{g\pm h} + \langle \sigma_o |\sigma_v=-1\rangle_{g\pm h}\Big) \notag \\
&\quad  + \Big(\mu_{g\pm h}(\sigma_v=1)-\frac{1}{2}\Big) \langle \sigma_o |\sigma_v=1\rangle_{g\pm h}  + \Big(\mu_{g\pm h}(\sigma_v=-1) -\frac{1}{2}\Big) \langle \sigma_o |\sigma_v=-1\rangle_{g\pm h} \notag \\
=\ & \frac{1}{2}\Big(\langle \sigma_o |\sigma_v=1\rangle_{g\pm h} + \langle \sigma_o |\sigma_v=-1\rangle_{g\pm h}\Big) +\frac{\langle \sigma_v\rangle_{g\pm h}}{2}\big(\langle \sigma_o |\sigma_v=1\rangle_{g\pm h} - \langle \sigma_o |\sigma_v=-1\rangle_{g\pm h}\big).  \notag
\end{align}
Let $c_\pm := \frac{1}{2}\big(\langle \sigma_o |\sigma_v=1\rangle_{g\pm h} - \langle \sigma_o |\sigma_v=-1\rangle_{g\pm h}\big)$, which are both non-negative by the FKG inequality for the Ising model (see \cite[Section 3.6.2]{FV17}). We then have
\begin{align}
& \langle \sigma_o\rangle_{g+ h} - \langle \sigma_o\rangle_{g -h} \notag \\
=\ &  \frac{1}{2}\Big(\langle \sigma_o |\sigma_v=1\rangle_{g+ h} + \langle \sigma_o |\sigma_v=-1\rangle_{g+ h}\Big) -  \frac{1}{2}\Big(\langle \sigma_o |\sigma_v=1\rangle_{g- h} + \langle \sigma_o |\sigma_v=-1\rangle_{g- h}\Big) \notag \\
& \qquad + \Big(c_+ \langle \sigma_v\rangle_{g+h} - c_- \langle \sigma_v\rangle_{g-h}\Big), \label{eq:sig6}
\end{align}
where we can apply Lemma \ref{L:ineq} to bound
\begin{equation}\label{eq:sig7}
\begin{aligned}
& c_+ \langle \sigma_v\rangle_{g+h} - c_- \langle \sigma_v\rangle_{g-h} \\
\leq\ &  c_+ \Big(\alpha \langle \sigma_v\rangle_{g+h}^{(H)} +(1-\alpha)\langle \sigma_v\rangle_{g+h}^{(\infty)}\Big) -c_-\Big(\alpha \langle \sigma_v\rangle_{g-h}^{(H)} +(1-\alpha)\langle \sigma_v\rangle_{g-h}^{(-\infty)}\Big),
\end{aligned}
\end{equation}
which replaces $\langle \sigma_v\rangle_{g+h}$ by a mixture of $\langle \sigma_v\rangle^{(H)}_{g+h}$ and $\langle \sigma_v\rangle^{(\infty)}_{g+h}$, and similarly for $\langle \sigma_v\rangle_{g-h}$. Substituting this bound into \eqref{eq:sig6} gives
\begin{align}
& \langle \sigma_o\rangle_{g+ h} - \langle \sigma_o\rangle_{g -h} \notag \\
\leq \ &  \frac{1}{2}\Big(\langle \sigma_o |\sigma_v=1\rangle_{g+ h} + \langle \sigma_o |\sigma_v=-1\rangle_{g+ h}\Big) -  \frac{1}{2}\Big(\langle \sigma_o |\sigma_v=1\rangle_{g- h} + \langle \sigma_o |\sigma_v=-1\rangle_{g- h}\Big) \notag \\
& \quad + c_+ \Big(\alpha \langle \sigma_v\rangle_{g+h}^{(H)} +(1-\alpha)\langle \sigma_v\rangle_{g+h}^{(\infty)}\Big) -c_-\Big(\alpha \langle \sigma_v\rangle_{g-h}^{(H)} +(1-\alpha)\langle \sigma_v\rangle_{g-h}^{(-\infty)}\Big) \notag \\
=\ & \alpha\Bigg\{  \frac{1}{2}\Big(\langle \sigma_o |\sigma_v=1\rangle_{g+ h} + \langle \sigma_o |\sigma_v=-1\rangle_{g+ h}\Big) -  \frac{1}{2}\Big(\langle \sigma_o |\sigma_v=1\rangle_{g- h} + \langle \sigma_o |\sigma_v=-1\rangle_{g- h}\Big) \notag \\
& \hspace{7cm} + c_+ \langle \sigma_v\rangle_{g+h}^{(H)} - c_- \langle \sigma_v\rangle_{g-h}^{(H)} \Bigg\} \label{eq:cpm1}\\
+ &  (1-\alpha)\Bigg\{  \frac{1}{2}\Big(\langle \sigma_o |\sigma_v=1\rangle_{g+ h} + \langle \sigma_o |\sigma_v=-1\rangle_{g+ h}\Big) -  \frac{1}{2}\Big(\langle \sigma_o |\sigma_v=1\rangle_{g- h} + \langle \sigma_o |\sigma_v=-1\rangle_{g- h}\Big) \notag \\
& \hspace{7cm} + c_+ \langle \sigma_v\rangle_{g+h}^{(\infty)} - c_- \langle \sigma_v\rangle_{g-h}^{(-\infty)} \Bigg\}. \label{eq:cpm2}
\end{align}
Note that $\langle \sigma_o |\sigma_v=\pm 1\rangle_{g\pm h}$ and $c_\pm$ do not depend on the external field at $v$. Therefore we can apply \eqref{eq:sig6}
to \eqref{eq:cpm1} (with $g_v=H$ and $h_v=0$) and \eqref{eq:cpm2} (with $H_v=\infty$) to rewrite the bound as
\begin{equation}
\begin{aligned}
\langle \sigma_o\rangle_{g+ h} - \langle \sigma_o\rangle_{g -h} & \leq \alpha \big(\langle \sigma_o\rangle_{g+h}^{(H)} - \langle \sigma_o\rangle_{g-h}^{(H)}\big) +(1-\alpha) \big(\langle \sigma_o\rangle_{g+h}^{(\infty)} - \langle \sigma_o\rangle_{g-h}^{(-\infty)}\big) \\
& \leq \alpha \big(\langle \sigma_o\rangle_{h}^{(0)} - \langle \sigma_o\rangle_{-h}^{(0)}\big) +(1-\alpha) \big(\langle \sigma_o\rangle_{h}^{(\infty)} - \langle \sigma_o\rangle_{-h}^{(-\infty)}\big) \\
& = \langle \sigma_o\rangle_h - \langle\sigma_o\rangle_{-h},
\end{aligned}
\end{equation}
where we applied \eqref{eq:2ineq} in the second inequality and \eqref{eq:sig3} in the equality. This concludes the proof of the induction step {\bf ($\dagger$)}. It only remains to prove Lemma \ref{L:ineq}, which will be carried out below.
\end{proof}
\medskip

\begin{proof}[Proof of Lemma \ref{L:ineq}] By scaling, it suffices to consider 
\begin{equation}\label{eq:ctheta}
c_+ = \sin^2\! \theta \qquad \mbox{and} \qquad c_-=\cos^2\!\theta \qquad \quad \mbox{for } \theta \in [0, \pi/2].
\end{equation}
Consider the measures $\mu_h^{(0)}$, $\mu_{g+h}^{(g_v)}$ and $\mu_{g-h}^{(g_v)}$, and let $z, x,y$ be respectively the effective field on $\sigma_v$ under these measures (compare with \eqref{eq:efield}). Namely,
\begin{equation}
\tanh z = \langle \sigma_v\rangle_h^{(0)}, \qquad  \tanh x = \langle \sigma_v\rangle_{g+h}^{(g_v)},  \qquad \tanh y = \langle \sigma_v\rangle_{g-h}^{(g_v)}.
\end{equation}
Note that $x,y\in \R$ because $g:V\to\R$, and $\tanh x \geq \tanh y$, i.e.,  $x\geq y$.

By the definition of $\alpha$ in \eqref{eq:sig2}, we have
$$
\tanh (z+h_v) = \alpha \tanh z + 1-\alpha
$$
Therefore
\begin{equation}
\alpha = \frac{1-\tanh (z+h_v)}{1-\tanh z} = \frac{1- \frac{\tanh z +\tanh h_v}{1+\tanh z\tanh h_v}}{1-\tanh z} = \frac{1-\tanh h_v}{1+\tanh z\tanh h_v},
\end{equation}
where we used $\tanh (a+b)= \frac{\tanh a +\tanh b}{1+\tanh a\tanh b}$.

Also observe that, by the induction hypothesis (i.e., the assumption in {\bf ($\dagger$)}), we have
\begin{equation}
\tanh x-\tanh y =  \langle \sigma_v\rangle_{g+h}^{(g_v)} -  \langle \sigma_v\rangle_{g-h}^{(g_v)} \leq \langle \sigma_v\rangle_{h}^{(0)} -  \langle \sigma_v\rangle_{-h}^{(0)} = 2\tanh z.
\end{equation}
The inequality \eqref{eq:ineq} can then be reformulated as showing that, under the constraints $x\geq y$ and $\tanh x-\tanh y \leq 2\tanh z$, we always have
\begin{equation}\label{eq:ineq1}
F(\theta, x,y,h_v):=\sin^2\!\theta\tanh (x+h_v)- \cos^2\!\theta \tanh (y-h_v) + \alpha (1-M_\theta) - 1 \leq 0,
\end{equation}
where
\begin{align}
M_\theta & := \sup_H \Big(\sin^2\!\theta\, \langle \sigma_v\rangle_{g+ h}^{(H)} - \cos^2\!\theta\, \langle \sigma_v\rangle_{g-h}^{(H)}\Big)
= \sup_H \Big(\sin^2\!\theta\, \langle \sigma_v\rangle_{g+ h}^{(g_v+H)} - \cos^2\!\theta\, \langle \sigma_v\rangle_{g-h}^{(g_v+H)}\Big) \notag \\
& = \sup_H \Big(\sin^2\!\theta\, \tanh (x+H) - \cos^2\!\theta\, \tanh(y+H)\Big). \label{eq:Mbeta}
\end{align}
Note that $1-M_\theta\geq 0$ and $\alpha$ is monotonically decreasing in $z$. To prove \eqref{eq:ineq1}, it suffices to consider the worst case scenario where $z$ takes its minimal value with $\tanh z=\frac{\tanh x-\tanh y}{2}$. In the definition of $F$ in \eqref{eq:ineq1}, we can therefore set
\begin{equation}\label{eq:alpha}
\alpha = \frac{1-\tanh h_v}{1+\frac{1}{2}(\tanh x-\tanh y)\tanh h_v} \in [0,1].
\end{equation}

To simplify notation, we make the following change of variables:
\begin{equation}\label{eq:abc}
a:= \tanh x, \qquad b:=\tanh y, \qquad c:=\tanh h_v,
\end{equation}
where $-1<b\leq a<1$ and $c\in [0,1]$. We can then rewrite $F$ in \eqref{eq:ineq1} as
\begin{equation}\label{eq:ineq2}
\begin{aligned}
F(\theta, a,b,c) &:= \sin^2\!\theta\, \frac{a+c}{1+ac} - \cos^2\!\theta\,\frac{b-c}{1-bc} + \frac{1-c}{1+\frac{1}{2}(a-b)c} (1-M_\theta) -1 \\
& = (1-c) \Big\{\frac{1-M_\theta}{1+\frac{1}{2}(a-b)c} - \Big(\sin^2\!\theta\, \frac{1-a}{1+ac} + \cos^2\!\theta\,\frac{1+b}{1-bc} \Big)\Big\},
\end{aligned}
\end{equation}
where we simplified the expression, using $1=\sin^2\!\theta +\cos^2\!\theta$. Note that $F=0$ when $c=1$. To show $F\leq 0$ for $c\in [0,1)$, it then suffices to show that
\begin{equation}\label{eq:ineq3}
\frac{1-M_\theta}{1+\frac{1}{2}(a-b)c} \leq  \sin^2\!\theta\, \frac{1-a}{1+ac} + \cos^2\!\theta\,\frac{1+b}{1-bc}.
\end{equation}
Denoting $d:=\tanh H$ and using $\tanh (x+H) = \frac{\tanh x +\tanh H}{1+\tanh x\tanh H}=\frac{a+d}{1+ad}$, we note that
\begin{equation}\label{eq:Mbeta2}
\begin{aligned}
1-M_\theta & = \inf_{d\in [-1, 1]} \Big(1-\sin^2\!\theta\, \frac{a+d}{1+ad} +\cos^2\!\theta\, \frac{b+d}{1+bd} \Big) \\
& = \inf_{d\in [-1, 1]} \Big(\sin^2\!\theta \, \frac{(1-a)(1-d)}{1+ad} +\cos^2\!\theta \, \frac{(1+b)(1+d)}{1+bd} \Big).
\end{aligned}
\end{equation}
Substituting this into \eqref{eq:ineq3}, multiplying both sides by $1+\frac{1}{2}(a-b)c>0$, and comparing the coefficients of $\sin^2\!\theta$ and $\cos^2\!\theta$, we note that \eqref{eq:ineq3}, and hence Lemma \ref{L:ineq}, would follow if for all $-1<b\leq a<1$, $c\in [0,1)$ and $\theta \in [0, \pi/2]$, we can find $d\in [-1,1]$ such that
\begin{equation}
\begin{aligned}
 \frac{(1-a)(1-d)}{1+ad} & \leq \Big(1+ \frac{1}{2} (a-b)c\Big) \frac{1-a}{1+ac},  \\
\frac{(1+b)(1+d)}{1+bd} & \leq \Big(1+ \frac{1}{2} (a-b)c\Big)  \frac{1+b}{1-bc},
\end{aligned}
\end{equation}
which is further equivalent to
\begin{equation}
\begin{aligned}
\frac{1+ac}{1+ \frac{1}{2} (a-b)c} & \leq \frac{1+ad}{1-d}, \\
\frac{1-bc}{1+ \frac{1}{2} (a-b)c} & \leq \frac{1+bd}{1+d}.
\end{aligned}
\end{equation}
Note that the two l.h.s.\ sum up to $2$. Also note that the two inequalities are invariant under the change of variables $(a,b, d)\to (-b,-a, -d)$, so we may assume w.l.o.g.\ that  $a+b\geq 0$, which implies $a\geq 0$ and the l.h.s.\ of the first inequality is greater than or equal to $1$. Note that as $d$ increases from $0$ to $1$, the r.h.s.\ of the first inequality increases from $1$ to $\infty$, and hence there exists $d\in [0,1]$ such that the first inequality becomes equality. The second inequality then follows if we show that for all $d\in [0,1]$, the sum of the two r.h.s.\ is greater than or equal to $2$. This is verified by the following calculation:
\begin{align*}
\frac{1+ad}{1-d} + \frac{1+bd}{1+d}  \geq 2 & \quad \Longleftrightarrow \quad  2+(a+b)d + (a-b)d^2 \geq 2(1-d^2) \\
& \quad \Longleftrightarrow \quad  (a+b)d + (2+ a-b)d^2 \geq 0,
\end{align*}
which holds when $a+b\geq 0$ and $d\in [0,1]$.
\end{proof}

\section{Proof of Corollaries \ref{C:corr} and \ref{C:mixing}}
\begin{proof}[Proof of Corollary \ref{C:corr}]
\blue{This follows from \eqref{eq:inf} by the observation that 
\begin{align*}
\langle \sigma_u \sigma_v\rangle_g -\langle \sigma_u\rangle_g \langle \sigma_v\rangle_g = \frac{\partial \langle \sigma_u\rangle_g}{\partial g_v} & =\lim_{a\downarrow 0}\frac{\langle \sigma_u\rangle_{g+a \delta_v}-\langle \sigma_u\rangle_{g-a\delta_v}}{2a} \\
& \leq \lim_{a\downarrow 0}\frac{\langle \sigma_u\rangle_{a \delta_v}-\langle \sigma_u\rangle_{-a\delta_v}}{2a} = \frac{\partial \langle \sigma_u\rangle_0}{\partial g_v}\Big|_{g_v=0}=\langle \sigma_u \sigma_v\rangle_0.
\end{align*}
}

We also give an alternative proof as follows. First we write the l.h.s.\ of \eqref{eq:corr} as
\begin{equation}\label{eq:siguv}
\begin{aligned}
& \langle (\sigma_u-\langle \sigma_u\rangle_g) \sigma_v\rangle_g \\
=\ &  \mu_g(\sigma_v=1)  \langle \sigma_u-\langle \sigma_u\rangle_g |\sigma_v=1\rangle_g - \mu_g(\sigma_v=-1)  \langle \sigma_u-\langle \sigma_u\rangle_g |\sigma_v=-1\rangle_g,
\end{aligned}
\end{equation}
where $\langle \cdot|\sigma_v=\pm 1\rangle_g$ denotes conditional expectation conditioned on $\sigma_v=\pm 1$. Note that
\begin{align*}
& \langle \sigma_u-\langle \sigma_u\rangle_g |\sigma_v=1\rangle_g  \\
=\ &  \langle \sigma_u |\sigma_v=1\rangle_g -\Big(\mu_g(\sigma_v=1) \langle \sigma_u|\sigma_v=1\rangle_g + \mu_g(\sigma_v=-1) \langle \sigma_u|\sigma_v=-1\rangle_g\Big) \\
=\ & \mu_g(\sigma_v=-1) \Big(\langle \sigma_u |\sigma_v=1\rangle_g  - \langle \sigma_u |\sigma_v=-1\rangle_g   \Big).
\end{align*}
Similarly, we find
\begin{align*}
\langle \sigma_u-\langle \sigma_u\rangle_g |\sigma_v=-1\rangle_g = -\mu_g(\sigma_v=1) \Big(\langle \sigma_u |\sigma_v=1\rangle_g  - \langle \sigma_u |\sigma_v=-1\rangle_g   \Big).
\end{align*}
Substituting these two identities into \eqref{eq:siguv} then gives
\begin{align}\label{eq:siguv2}
\langle \sigma_u \sigma_v\rangle_g  - \langle \sigma_u\rangle_g \langle \sigma_v\rangle_g& = 2 \mu_g(\sigma_v=1) \mu_g(\sigma_v=-1) \Big(\langle \sigma_u |\sigma_v=1\rangle_g  - \langle \sigma_u |\sigma_v=-1\rangle_g   \Big).
\end{align}
Note that $\mu_g(\sigma_v=1) \mu_g(\sigma_v=-1) \leq 1/4 = \mu_0(\sigma_v=1) \mu_0(\sigma_v=-1)$, while
$$
\langle \sigma_u |\sigma_v=1\rangle_g  - \langle \sigma_u |\sigma_v=-1\rangle_g  \leq \langle \sigma_u |\sigma_v=1\rangle_0  - \langle \sigma_u |\sigma_v=-1\rangle_0
$$
by Theorem \ref{T:inf} with $h=\infty$ at $v$ and $h=0$ elsewhere. These two bounds then imply \eqref{eq:corr} since \eqref{eq:siguv2} also holds when $g\equiv 0$, where $\langle \sigma_u\rangle_0 =\langle \sigma_v\rangle_0=0$ by symmetry.
\end{proof}
\medskip

\begin{proof}[Proof of Corollary \ref{C:mixing}] Fix a finite set $V\subset \Z^d$ for some $d\geq 2$, with external boundary $\partial V$. Let $\beta<\beta_c$, and choose any $h: V\to \R$ and boundary condition $\tau: \partial V \to \{\pm 1\}$. Let $y\in \partial V$ be arbitrary, and $\tau^y$ be obtained from $\tau$ by flipping the spin $\tau_y$. Let $\Delta \subset V$ be arbitrary.

Let $R:= \lfloor d(y, \Delta)/2\rfloor$, and let $S:=\{ x\in \Z^d: |x-y|_1=R\}$, the $l^1$ sphere of radius $R$ centered at $y$, which separates $y$ from $\Delta$. We may assume $R$ is sufficiently large, otherwise \eqref{eq:mixing} holds by choosing a large prefactor $C$.  First note that
$$
\Vert \mu^{\tau,h}_{V, \beta}|_\Delta - \mu^{\tau^y, h}_{V, \beta}|_\Delta\Vert_{TV} \leq \Vert \mu^{\tau,h}_{V, \beta}|_{S\cap V} - \mu^{\tau^y, h}_{V, \beta}|_{S\cap V}\Vert_{TV},
$$
because $\Delta$ is contained in a subset of $V$ with boundary spins $(\tau_v)_{|v-y|_1\geq R}$ and $(\sigma_v)_{v\in S\cap V}$, and conditioned on  $(\sigma_v)_{v\in S\cap V}$, the law of $(\sigma_v)_{v\in \Delta}$ is the same under both $\mu^{\tau,h}_{V, \beta}$ and $\mu^{\tau^y,h}_{V, \beta}$. Therefore it suffices to prove \eqref{eq:mixing} with $S\cap V$ in place of $\Delta$.

Label the vertices in $S\cap V$ by $v_1, \ldots, v_n$, with $n:=|S\cap V|$. We proceed by successively coupling the spins $\sigma_{v_1}, \ldots, \sigma_{v_n}$ under the two measures $\mu^{\tau,h}_{V, \beta}$ and $\mu^{\tau^y,h}_{V, \beta}$. By Theorem \ref{T:inf},
\begin{equation}\label{eq:couple1}
\begin{aligned}
\big|\langle \sigma_{v_1} \rangle^{\tau, h}_{V, \beta} - \langle \sigma_{v_1} \rangle^{\tau^y, h}_{V, \beta} \big| \leq \langle \sigma_{v_1} \rangle^{y+}_{V, \beta} - \langle \sigma_{v_1} \rangle^{y-}_{V, \beta} & = 2\langle \sigma_y \sigma_{v_1}\rangle_{V\cup \{y\}, \beta} \\
& \leq 2\langle \sigma_y \sigma_{v_1}\rangle_{\Z^d, \beta} \leq C_1 e^{-C_2 R},
\end{aligned}
\end{equation}
where $\langle \cdot \rangle^{y\pm}_{V, \beta}$ denotes expectation w.r.t.\ the Ising measure with zero external field and free boundary condition, except for a single boundary spin at $y\in \partial V$ with value $\pm 1$; $\langle \cdot \rangle_{\Lambda, \beta}$ corresponds to the free boundary condition in a domain $\Lambda$ with zero external field; the second inequality in \eqref{eq:couple1} is a consequence of the GKS inequality \cite[Section 3.6.1]{FV17} since the coupling constants $J$ over the edges connecting $V\cup \{y\}$ and $(V\cup \{y\})^c$ are increased from $0$ to $1$; the last inequality in \eqref{eq:couple1} holds because $\beta<\beta_c$~\cite{ABF87}. Since $\langle \sigma_v\rangle = 2 \mu(\sigma_v=1)-1$, \eqref{eq:couple1} implies
\begin{equation}\label{eq:couple2}
\big|\mu^{\tau, h}_{V, \beta}(\sigma_{v_1}=1)  -\mu^{\tau^y, h}_{V, \beta}(\sigma_{v_1}=1)\big| \leq \frac{C_1}{2} e^{-C_2 R}.
\end{equation}
We can therefore construct a coupling $(\sigma_{v_1}, \sigma^y_{v_1})$ such that the marginal laws of the two components equal that of the spin at $v_1$
under $\mu^{\tau, h}_{V, \beta}$ and $\mu^{\tau^y, h}_{V, \beta}$, respectively, while $\P(\sigma_{v_1}\neq \sigma^y_{v_1})\leq C_1 e^{-C_2 R}$.

Conditioned on $(\sigma_{v_1}, \sigma^y_{v_1})$ with $\sigma_{v_1}=\sigma^y_{v_1}=:\tau_{v_1}$, we can then couple $(\sigma_{v_2}, \sigma^y_{v_2})$ such that the marginal laws of the two components equal that of the spin at $v_2$, conditioned on the spin at $v_1$ being equal to $\tau_{v_1}$ and under $\mu^{\tau, h}_{V, \beta}$ and $\mu^{\tau^y, h}_{V, \beta}$, respectively. This conditioning effectively imposes an additional boundary condition $\tau_{v_1}$ at $v_1$. But the new boundary conditions $\{\tau, \tau_{v_1}\}$ and $\{\tau^y, \tau_{v_1}\}$ still differ only at $y$. So the calculations leading to \eqref{eq:couple2} remains valid, which implies that we can define a coupling of $(\sigma_{v_2}, \sigma^y_{v_2})$ such that
$$
\P(\sigma_{v_2}\neq \sigma^y_{v_2} | (\sigma_{v_1}, \sigma^y_{v_1}))\cdot 1_{\{\sigma_{v_1}=\sigma^y_{v_1}\}} \leq C_1e^{-C_2R}.
$$
We can then iterate this procedure to obtain a coupling between $(\sigma_v)_{v\in S\cap V}$ and $(\sigma^y_v)_{v\in S\cap V}$, whose laws equal that of the spins in $S\cap V$ under $\mu^{\tau, h}_{V, \beta}$ and $\mu^{\tau^y, h}_{V, \beta}$, respectively, while
$$
\P((\sigma_v)_{v\in S\cap V}= (\sigma^y_v)_{v\in S\cap V}) \geq (1-C_1 e^{-C_2 R})^{|S\cap V|} \geq 1-C_1 |S\cap V| e^{-C_2R}.
$$
Since $|S\cap V|\leq c R^{d-1}= c e^{(d-1)\log R}$, this implies \eqref{eq:mixing} with $S\cap V$ in place of $\Delta$.
\end{proof}

\appendix

\section{Some Counterexamples}\label{S:example} 
In this section, we give the counterexamples mentioned in Remarks \ref{rem1} and \ref{rmk1.4}.

\blue{
\begin{proof}[Counterexample for Remark \ref{rem1}]
Consider $V=\{-2,-1,0,1,2\}$ with edges between neighbouring integers and $J_e=1$ for every edge $e$. Let $g_{-2}=g_{-1}=-2$ and $g_0=0$. It is possible to choose $g_1<2<g_2$ such that the effective field on $\sigma_0$ induced by spins to its left $($cf.~\eqref{eq:efi} below$)$ exactly cancels out the field induced by spins to its right.  If we choose $h_i= 1_{\{i=0\}}$, then $\langle \sigma_0\rangle_{g+h}-\langle\sigma_0\rangle_{g-h}=\langle \sigma_0\rangle_{h}-\langle\sigma_0\rangle_{-h}$ already achieves the maximum. Changing $g$ to $\lambda g$, say for $\lambda \in (0,1)$, in general breaks the balance between the effective fields induced on $\sigma_0$ by spins to its left and right, which leads to strictly smaller values of $\langle \sigma_0\rangle_{\lambda g+h}-\langle\sigma_0\rangle_{\lambda g-h}$ $($see \eqref{eq:gpmh}$)$.
\end{proof}
}

\blue{
\begin{proof}[Counterexample for Remark \ref{rmk1.4}]
Consider a tree with root $u$ and three leaves $v, a, b$. Assume that $J_{ua}\in (0,1)$ and
$J_{uv}=J_{ub}=1$. Consider $\tilde g$ with $\tilde g_u= \tilde g_v=0$ and $\tilde g_b=1$. Then we can find $\tilde g_a<-1$ such that the effective field on $\sigma_u$
induced by $\sigma_a$ $($cf.~\eqref{eq:efi}$)$ exactly cancels out the field induced by $\sigma_b$, which implies that 
equality holds in \eqref{eq:corr} with $g=\tilde g$. If we replace $g=\lambda \tilde g$, then because the effective fields induced by $\sigma_a$ and $\sigma_b$ 
on $\sigma_u$ are distinct non-linear functions of $\lambda$, we can find $\lambda_0\in (0,1)$ such that $\sigma_a$ and $\sigma_b$ together induce a non-zero effective field on $\sigma_u$. Under the external field $\lambda_0 \tilde g$, inequality in \eqref{eq:corr} can be seen to be strict. This implies that the l.h.s.\ of \eqref{eq:corr} with $g=\lambda \tilde g$ is not monotonically decreasing in $\lambda\geq 0$. We can also construct an example with $J\equiv 1$ by inserting a vertex $\tilde a$ between $u$ and $a$. This has the same effect as having an effective coupling $J_{ua}\in (0,1)$ between $u$ and $a$ with $e^{2J_{ua}}= \cosh 2,$\footnote[1]{The effective coupling constant $J_{ua}$ can be identified by first setting $g \equiv 0$ and then computing the ratio of the Gibbs weights associated to the four sets
of configurations with $(\sigma_u, \sigma_a)=(\pm1,\pm1)$. See \eqref{eq:efi} for a similar computation of the effective field.} which is just the example we already constructed.
\end{proof}
}

\section*{Acknowledgements}
J.~Ding wishes to thank Yuval Peres for interesting discussions during his Ph.D.\ studies concerning the validity of Theorem 1.1 and pointing out that it is connected to a conjecture of Shlosman's, and he wishes to thank David Gamarnik for an interesting discussion on the strong spatial mixing property. We thank Ronen Eldan, Trishen Gunaratnam,
Kuikui Liu, Elchanan Mossel, Charles M.~Newman, Yuval Peres, Akira Sakai, Barry Simon, Rob van den Berg, and the referees for helpful comments on the paper. J.~Ding is partially supported by NSF grant DMS-1757479 and DMS-1953848. Much of the work was carried out when J.~Ding was a faculty member of the University of Pennsylvania. J. Song is partially supported by Shandong University grant 11140089963041 and National Natural Science Foundation of China grant 12071256. R. Sun is supported by NUS grant R-146-000-288-114.

\end{document}